\documentclass[a4paper, 11pt]{amsart}
\usepackage[
  margin=1.9cm,
  includefoot,
  footskip=30pt,
]{geometry}
\usepackage{amsmath,amssymb,amscd}
\usepackage{amsfonts}
\usepackage[english]{babel}
\usepackage[leqno]{amsmath}
\usepackage{amssymb,amsthm}
\usepackage{mathrsfs}
\usepackage{amscd}
\usepackage{enumerate}
\usepackage{epsfig}
\usepackage{relsize}
\usepackage{layout}
\usepackage[usenames,dvipsnames]{xcolor}
\usepackage[backref=page]{hyperref}
\usepackage{tikz-cd}
\usepackage{tikz}
\usepackage{tikz-3dplot}
\hypersetup{
 colorlinks,
 citecolor=Green,
 linkcolor=Red,
 urlcolor=Blue}
\definecolor{gray(x11gray)}{rgb}{0.75, 0.75, 0.75}
\definecolor{aliceblue}{rgb}{0.94, 0.97, 1.0}
\usepackage[matrix,arrow,tips,curve]{xy}
\input{xy}
\xyoption{all}
\definecolor{darkgreen}{rgb}{0.0, 0.7, 0.0}
\definecolor{purple}{rgb}{0.5, 0.0, 0.5}
\definecolor{red}{rgb}{0.8, 0.2, 0.0}

\newtheorem{thm}{Theorem}[section]
\newtheorem{bthm}{Theorem}

\newtheorem{lemma}[thm]{Lemma}

\newtheorem{claim}[thm]{Claim}

\numberwithin{equation}{section}
\theoremstyle{definition}
\newtheorem{defi}[thm]{Definition}

\newtheorem{notation}[thm]{Notation}

\theoremstyle{remark}

\newtheorem{example}[thm]{Example}
\newcommand{\Z}{\mathbb{Z}}

\def \P{\mathbb{P}}

\def \F{\mathcal F}

\def\I{{\mathcal J}}
\def \L{\mathcal L}
\def \E{\mathcal E}
\def \G{\mathcal G}

\def\O{\mathcal O}

\def\M0{\mathcal M^0}

\newcommand{\Id}{{\operatorname{Id}}}

\DeclareMathOperator{\Ker}{{Ker}}

\newcommand{\rk}{\operatorname{rk}}

\title[On varieties with Ulrich twisted conormal bundles]{On varieties with Ulrich twisted conormal bundles}

\author[V. Antonelli, G. Casnati, A.F. Lopez, D. Raychaudhury]{Vincenzo Antonelli, Gianfranco Casnati, Angelo Felice Lopez and Debaditya Raychaudhury}

\address{\hskip -.43cm Vincenzo Antonelli, Dipartimento di Scienze Matematiche, Politecnico di Torino, c.so Duca degli Abruzzi 24,
10129 Torino, Italy. email: {\tt vincenzo.antonelli@polito.it}}

\address{\hskip -.43cm Gianfranco Casnati, Dipartimento di Scienze Matematiche, Politecnico di Torino, c.so Duca degli Abruzzi 24,
10129 Torino, Italy. email: {\tt gianfranco.casnati@polito.it}}

\address{\hskip -.43cm Angelo Felice Lopez, Dipartimento di Matematica e Fisica, Universit\`a di Roma
Tre, Largo San Leonardo Murialdo 1, 00146, Roma, Italy. e-mail {\tt angelo.lopez@uniroma3.it}}

\address{\hskip -.43cm Debaditya Raychaudhury, Department of Mathematics, University of Arizona, 617 N Santa Rita Ave., Tucson, AZ 85721, USA. email: {\tt draychaudhury@math.arizona.edu}}

\thanks{The first three authors are members of the GNSAGA group of INdAM. The third author was partially supported by PRIN ``Advances in Moduli Theory and Birational Classification''.}

\thanks{{\it Mathematics Subject Classification} : Primary 14F06. Secondary 14H60, 14H10, 14J60.}

\begin{document}

\begin{abstract} 
We study varieties $X \subset \P^r$ such that $N_X^*(k)$ is an Ulrich vector bundle for some integer $k$. We first prove that such an $X$ must be a curve. Then we give several examples of curves with $N_X^*(k)$ an Ulrich vector bundle.
\end{abstract}

\maketitle

\section{Introduction}

Let $X \subset \P^r$ be a smooth variety of dimension $n \ge 1$ over an algebraically closed field $F$. Recall that a vector bundle $\E$ on $X$ is called Ulrich if $H^i(\E(-p))=0$ for all $i \ge 0$ and $1 \le p \le n$. The importance of Ulrich vector bundles is well-known (see for example \cite{es, b1, cmp} and references therein). While the main general problem about Ulrich vector bundles is their conjectural existence, another line of research around them is what are the consequences on the geometry of $X$ in the presence of an Ulrich vector bundle. In this vein, we continue our study of which natural bundles, associated to $X$ and to its embedding in $\P^r$, can be Ulrich up to some twist. 

In previous papers, the third and fourth authors analyzed normal and tangent bundles, see \cite{l, lr} (see also \cite{bmpt}). In the present paper we study the following question: for which integers $k$ one has that the twisted conormal bundle $N_X^*(k)$ is Ulrich?

A first simple consequence can be drawn: if $X$ is degenerate, that is contained in a hyperplane, then $(X,\O_X(1),k)=(\P^n,\O_{\P^n}(1),1)$, see Lemma \ref{deg}.

On the other hand, suppose that $X$ is nondegenerate. While in previous cases \cite{l}, \cite{lr}, examples of surfaces and threefolds appeared, we find a very different result for the conormal bundle. In fact we show that the answer to the above question is negative in dimension at least two.

\begin{bthm} (char(F)=0)
\label{n<2}

\hskip 3cm

Let $X \subset \P^r$ be a smooth nondegenerate variety such that $N_X^*(k)$ is Ulrich. Then $X$ is a curve.
\end{bthm}

Now, for curves the situation is wide. First of all, there are many examples, at least in $\P^3$, stemming from some classical works \cite{el, eh, be} (see Examples \ref{elli} and \ref{sottoc}). 

We first prove that there is a bound, sharp in codimension $2$, for the degree of a curve having Ulrich twisted conormal bundle.

\begin{bthm} 
\label{sh}

\hskip 3cm

Let $C \subset \P^{c+1}$ be a smooth nondegenerate curve of degree $d$ and codimension $c \ge 1$ such that $N_C^*(k)$ is Ulrich. Then $c \ge 2$ and
\begin{equation}
\label{bd}
d \ge \frac{c+2}{2k+c}\binom{k+c}{c+1}.
\end{equation}
Moreover this bound is sharp for $c=2$ and $k \equiv 1, 3 \ ({\rm mod} \ 6)$. 
\end{bthm}
On the other hand, the examples mentioned above, Examples \ref{elli} and \ref{sottoc}, are all subcanonical curves in $\P^3$.  We show that neither the fact of being subcanonical, nor of lying in $\P^3$, is a necessary condition, by producing examples, for unbounded genus, of non-subcanonical curves in $\P^3$ and in $\P^4$.

\begin{bthm}  
\label{nuovo2}

\hskip 3cm

\noindent (i) Let $X \subset \P^3$ be a general nonspecial curve of genus $g$ and degree $d=2g-2$. Then $N_X^*(4)$ is Ulrich and $X$ is not subcanonical.

\noindent (ii) Let $X \subset \P^4$ be a general curve of genus $g \ge 3$ and degree $d=5g-5$. Then $N_X^*(3)$ is Ulrich and $X$ is not subcanonical.
\end{bthm}
Note that the nonspecial curve of genus $5$ and degree $8$ in (i) of the above theorem, is a non-subcanonical curve realizing equality in Theorem \ref{sh}. 
 
Here is a brief overview of the strategy of proofs of the theorems. 

Assume that $N_X^*(k)$ is Ulrich. The first ingredient (Lemma \ref{coh} (iv)) is that $X$ is contained in a hypersurface of degree $k$. Next, some vanishings hold for the cohomology of twists of $\O_X$ (Lemma \ref{coh} (i)). In the case of curves, using Riemann-Roch this gives the bound in Theorem \ref{sh}. In the case of higher dimensional varieties, since the property of $N_X^*(k)$ being Ulrich is preserved by hyperplane sections, one reduces to the surface section $S$. Also, $S$ is projectively normal (Lemma \ref{coh} (ii)) and combining again the vanishings, Riemann-Roch and the Bogomolov inequality (because Ulrich bundles are semistable), one contradicts Theorem \ref{sh}. As for Theorem \ref{nuovo2}, we use degenerations of nonspecial curves and the techniques of interpolation as in \cite{aly, lv}.

We would like to thank the generous referee for the big contribution given to improve the hypotheses and shorten the proof of Theorem \ref{nuovo2}. 
%We believe that many more examples of curves with $N_X^*(k)$ Ulrich, can probably be constructed, by refining the methods, such as using elementary modifications and special degenerations as in \cite{aly, br, r}. One possible direction is given in Remark \ref{gennsp}.

\section{Notation}

Throughout the paper we work over an algebraically closed field $F$. In some cases, when needed, we will specify that ${\rm char}(F)=0$. A {\it variety} is by definition an integral separated scheme of finite type over $F$. A {\it curve} (respectively a {\it surface}) is a variety of dimension $1$ (resp. $2$). Moreover, we henceforth establish the following:

\begin{notation}
\label{not}

\hskip 3cm

\begin{itemize}
\item $X \subset \P^r$ is a smooth closed variety of dimension $n \ge 1$ and codimension $c=r-n \ge 1$.
\item $H$ is a hyperplane divisor.
\item $N_X := N_{X/\P^r}$ is the normal bundle. 
\item For any sheaf $\G$ on $X$ we set $\G(l)=\G(lH)$.
\item $d=H^n$ is the degree of $X$.
\item $C$ is a general curve section of $X$ under $H$.
\item $S$ is a general surface section of $X$ under $H$, when $n \ge 2$.
\item $g=g(C)=\frac{1}{2}[K_X H^{n-1}+(n-1)d]+1$ is the sectional genus of $X$.
\item For $1 \le i \le n-1$, let $H_i \in |H|$ be general divisors and set $X_n:=X$ and $X_i=H_1\cap\cdots\cap H_{n-i}$. In particular $X_1=C, X_2=S$. 
\item $s(X)= \min\{s \ge 1: H^0(\I_{X/\P^r}(s)) \ne 0\}$.
\end{itemize}
\end{notation}

We will also let $V=H^0(\O_{\P^r}(1))$ and consider the exact sequences
\begin{equation}
\label{eul}
0 \to {\Omega^1_{\P^r}}_{|X} \to V \otimes \O_X(-1) \to \O_X \to 0
\end{equation}
and
\begin{equation}
\label{nor}
0 \to N_X^* \to {\Omega^1_{\P^r}}_{|X} \to \Omega^1_X \to 0.
\end{equation}

\section{A general fact about projective varieties}

We record here a simple but useful fact.

\begin{lemma}
\label{s}
Let $X \subset \P^r$ be a smooth variety of dimension $n \ge 1$. If $H^0(N_X^*(l))=0$ and $\pi : X \to \overline{X} \subset \P^m$ is an isomorphic projection, then $l \le \min\{s(X)-1, s(\overline{X})-1\}$. 
\end{lemma}
\begin{proof}
Set $s=s(X)$ and suppose that $l \ge s$, so that $H^0(N_X^*(s))=0$. Now the exact sequence
$$0 \to \I^2_{X/\P^r}(s) \to \I_{X/\P^r}(s) \to N_X^*(s) \to 0$$
implies that $h^0(\I^2_{X/\P^r}(s))=h^0(\I_{X/\P^r}(s))>0$, hence there is a hypersurface $G$ of degree $s$ such that $X \subseteq {\rm Sing}(G)$. Hence there is a non-zero partial derivative of the equation of $G$, giving a hypersurface of degree $s-1$ containing $X$, a contradiction. Therefore 
\begin{equation}
\label{primo}
l \le s(X)-1.
\end{equation}
Now let $\pi : X \to \overline{X} \subset \P^m$ be an isomorphic projection, so that we have 
%an exact diagram
%$$\xymatrix{& & 0 \ar[d] & 0 \ar[d] & \\ & & \O_X(1)^{\oplus (r-m)} \ar[d] \ar[r]^{\cong} & \O_X(1)^{\oplus (r-m)} \ar[d] & & \\ 0 \ar[r] & \O_X \ar[r]  \ar[d]^{\cong} & \O_X(1)^{\oplus (r+1)} \ar[r]  \ar[d] & {T_{\P^r}}_{|X} \ar[r] \ar[d] & 0 \\ 0 \ar[r] & \pi^*\O_{\overline{X}} \ar[r]  & \pi^*\O_{\overline{X}}(1)^{\oplus (m+1)} \ar[r] \ar[d] & \pi^*{T_{\P^m}}_{|\overline{X}} \ar[r] \ar[d] & 0 \\ & & 0 & 0 & }$$
%and therefore also an exact diagram
%$$\xymatrix{& & 0 \ar[d] & 0 \ar[d] & \\ & & \O_X(1)^{\oplus (r-m)} \ar[d] \ar[r]^{\cong} & \O_X(1)^{\oplus (r-m)} \ar[d] & & \\ 0 \ar[r] & T_X \ar[r]  \ar[d]^{\cong} & {T_{\P^r}}_{|X} \ar[r]  \ar[d] & N_X \ar[r] \ar[d] & 0 \\ 0 \ar[r] & \pi^*T_{\overline{X}} \ar[r] & \pi^*{T_{\P^m}}_{|\overline{X}} \ar[r] \ar[d] & \pi^*N_{\overline{X}/\P^m} \ar[r] \ar[d] & 0 \\ & & 0 & 0 & }.$$
%Hence, we deduce 
an exact sequence
\begin{equation}
\label{proj}
0 \to \pi^*N_{\overline{X}/\P^m}^* \to N_X^* \to \O_X(-1)^{\oplus (r-m)} \to 0.
\end{equation}
Since $H^0(N_X^*(l))=0$ we get that $H^0(\overline{X},N_{\overline{X}/\P^m}^*(l))=H^0(X,\pi^*N_{\overline{X}/\P^m}^*(l))=0$. Hence applying \eqref{primo} to $\overline{X} \subset \P^m$ we get that $l \le s(\overline{X})-1$ and the lemma is proved. 
\end{proof}

\section{Generalities on Ulrich bundles}

We collect here some well-known facts about Ulrich bundles, to be used sometimes later.

\begin{defi}
Let $\E$ be a vector bundle on $X$. We say that $\E$ is {\it Ulrich} for $(X,H)$ if $H^i(\E(-p))=0$ for all $i \ge 0$ and $1 \le p \le n$.
\end{defi}

We have

\begin{lemma}
\label{ulr}
Let $\E$ be a rank $t$ Ulrich vector bundle for $(X,H)$. Then
\begin{itemize}
\item[(i)] $c_1(\E) H^{n-1}=\frac{t}{2}[K_X+(n+1)H] H^{n-1}$.
\item[(ii)] $\E^*(K_X+(n+1)H)$ is also Ulrich for $(X,H)$.
\item [(iii)] $\E$ is globally generated.
\item [(iv)] $\E$ is arithmetically Cohen-Macaulay (aCM), that is $H^i(\E(j))=0$ for $0 < i <n$ and all $j \in \Z$.
\item [(v)] $\E_{|Y}$ is Ulrich on a smooth hyperplane section $Y$ of $X$.
\item [(vi)] $\O_X(l)$ is Ulrich if and only if $(X,H,l)=(\P^n,\O_{\P^n}(1),0)$.
\end{itemize}
\end{lemma}
\begin{proof}
Well-known. For (i)-(v) see for example \cite[Lemma 3.2]{lr}. As for (vi), it is obvious that $\O_{\P^n}$ is Ulrich for $(\P^n,\O_{\P^n}(1))$. Vice versa, if $\O_X(l)$ is Ulrich, then it is globally generated by (iii), so that $l \ge 0$. But also $H^0(\O_X(l-1))=0$, hence $l=0$. It follows by \cite[Lemma 3.2(vii)]{lr} that $d=h^0(\O_X)=1$, so that $(X,H)=(\P^n,\O_{\P^n}(1))$.
\end{proof}

\begin{lemma}
\label{ulr2}
Let $X \subset \P^r$ be a smooth variety of dimension $n \ge 3$. Let $\E$ be a vector bundle on $X$ and let $Y$ be a smooth hyperplane section of $X$. If $\E_{|Y}$ is Ulrich, then $\E$ is Ulrich.
\end{lemma}
\begin{proof}
For $j \in \Z$ consider the exact sequence
\begin{equation}
\label{sez}
0 \to \E(j-1) \to \E(j) \to \E_{|Y}(j) \to 0.
\end{equation}
If $2 \le i \le n-2$ we have that $H^{i-1}(\E_{|Y}(j))=H^i(\E_{|Y}(j))=0$ for any $j \in \Z$ by Lemma \ref{ulr}(iv). Hence \eqref{sez} gives that $h^i(\E(j-1))=h^i(\E(j))$ for any $j \in \Z$. On the other hand $h^i(\E(j))=0$ for $j \gg 0$ and it follows that $h^i(\E(j))=0$ for any $j \in \Z$ and $2 \le i \le n-2$.

Suppose now that $i \in \{0,1\}$ and $j \le -1$. We have that $H^0(\E_{|Y}(j))=0$ and, since $n-1 \ge 2$, also that $H^1(\E_{|Y}(j))=0$ by Lemma \ref{ulr}(iv). Hence \eqref{sez} gives that $h^i(\E(j-1))=h^i(\E(j))$. On the other hand, by Serre duality, $h^i(\E(j))=h^{n-i}(\E^*(K_X-jH))=0$ for $j \ll 0$ and therefore
\begin{equation}
\label{van}
h^i(\E(j))=0 \ \hbox{for} \ i \in \{0,1\} \ \hbox{and} \ j \le -1.
\end{equation}
Now let $\E'=\E^*(K_X+(n+1)H)$. Then  $\E'_{|Y}=\E_{|Y}^*(K_Y+nH_{|Y})$ is also Ulrich by Lemma \ref{ulr}(ii). Therefore \eqref{van} implies that $h^i(\E'(j))=0$ for $i \in \{0,1\}$ and $j \le -1$. By Serre duality we get that
$h^{n-i}(\E(-n-1-j))=h^i(\E^*(K_X+(n+1+j)H))=h^i(\E'(j))=0$ for $i \in \{0,1\}$ and $j \le -1$. But this is the same as $h^s(\E(l))=0$ for $s \in \{n-1,n\}$ and $l \ge -n$. 

Thus we have proved that $H^i(\E(-p))=0$ for $i \ge 0$ and $1 \le p \le n$, that is $\E$ is Ulrich.
\end{proof}

\section{Ulrich conormal bundles}

In this section we will draw some very useful consequences and facts for varieties $X \subset \P^r$ such that $N_X^*(k)$ is Ulrich. 

The first one is a reduction via hyperplane sections (for the $X_i$'s see Notation \ref{not}).

\begin{lemma}
\label{sezi}
Let $X \subset \P^{c+n}$ be a smooth variety of dimension $n$ and codimension $c \ge 1$. If $n \ge 2$ and $N_X^*(k)$ is Ulrich, then $N_{X_i/\P^{c+i}}^*(k)$ is Ulrich for all $i \in \{1,\ldots,n-1\}$. Vice versa, if $n \ge 3$ and $N_{X_i/\P^{c+i}}^*(k)$ is Ulrich for some $i \in \{2,\ldots,n-1\}$, then $N_{X_j/\P^{c+j}}^*(k)$ is Ulrich for all $j \in \{2,\ldots,n\}$ (hence in particular so is $N_X^*(k)$).
\end{lemma}
\begin{proof}
Recall that if $Y \subset \P^m$ is smooth and $Z$ is a smooth hyperplane section, then 
$$(N_{Y/\P^m})_{|Z} \cong N_{Z/\P^{m-1}}.$$ 
Now if $N_X^*(k)$ is Ulrich, then so are all $N_{X_i/\P^{c+i}}^*(k)$ by Lemma \ref{ulr}(v). 

Vice versa, if $N_{X_i/\P^{c+i}}^*(k)$ is Ulrich for some $i \in \{2,\ldots,n-1\}$, then $N_{X_{i+1}/\P^{c+i+1}}^*(k)$ is Ulrich by Lemma \ref{ulr2}. Repeating the argument we get that $N_{X_j/\P^{c+j}}^*(k)$ is Ulrich for all $j \in \{2,\ldots,n\}$.
\end{proof}

We now deal with $X$ degenerate in $\P^r$.

\begin{lemma}
\label{deg}
Let $X \subset \P^r$ be a smooth degenerate variety of dimension $n \ge 1$. Then $N_X^*(k)$ is Ulrich if and only if $(X,H,k)=(\P^n,\O_{\P^n}(1),1)$.
\end{lemma}
\begin{proof}
If $(X,H)=(\P^n,\O_{\P^n}(1))$ then $N_X^*(1)=\O_{\P^n}^{\oplus c}$ is Ulrich. 

Vice versa assume that $N_X^*(k)$ is Ulrich. Since $X$ is degenerate, $N_X^*(k)$ has $\O_X(k-1)$ as a direct summand. Therefore also $\O_X(k-1)$ is Ulrich and Lemma \ref{ulr}(vi) gives that $(X,H,k)=(\P^n,\O_{\P^n}(1),1)$.
\end{proof}

In the sequel we will then consider only nondegenerate varieties.

We start by collecting some cohomological and numerical conditions.

\begin{lemma} {\rm (cohomological conditions)} 
\label{coh} 

Let $X \subset \P^{c+n}$ be a smooth nondegenerate variety of dimension $n \ge 1$ and codimension $c \ge 1$. If $N_X^*(k)$ is Ulrich,  we have: 
\begin{itemize}
\item[(i)] $H^n(\O_X(l))=0$ for every $l \ge k-n-1$.
\item[(ii)] If $n \ge 2$, then $X \subset \P^{c+n}$ is projectively normal. 
\item[(iii)] If $n \ge 2$ then $q(X)=0$.
\item[(iv)] $k \le \min\{s(X), s(\overline{X})\}$, where $\overline{X} \subset \P^r$ is any isomorphic projection of $X$.
\end{itemize} 
\end{lemma}
\begin{proof} 
By hypothesis $N_X^*(k)$ is Ulrich, hence it is aCM by Lemma \ref{ulr}(iv). 

To see (i) we just need to prove that $H^n(\O_X(k-n-1))=0$. Assume that $H^n(\O_X(k-n-1)) \ne 0$, that is, by Serre duality, $H^0(K_X+(n+1-k)H) \ne 0$. Then we have an inclusion 
$$N_X(-1) \hookrightarrow N_X(K_X+(n-k)H)$$ 
and, since  $N_X(-1)$ is globally generated, we get that $h^0(N_X(K_X+(n-k)H)) \ne 0$. On the other hand, $N_X(K_X+(n+1-k)H)$ is Ulrich by Lemma \ref{ulr}(ii), hence $h^0(N_X(K_X+(n-k)H)) = 0$. This contradiction proves (i). To see (ii), let $V=H^0(\O_{\P^{c+n}}(1))$ and let $P^1(\O_X(1))$ be the sheaf of principal parts and consider, as in \cite[Proof of Thm.~2.4]{ei}, the following commutative diagram
$$\xymatrix{& 0 \ar[d] & 0 \ar[d] & & \\ & N_X^*(1) \ar[r]^{\cong} \ar[d] & N_X^*(1) \ar[d] & & \\ 0 \ar[r] & \Omega^1_{\P^{c+n}}(1)_{|X} \ar[r]  \ar[d] & V \otimes \O_X \ar[r]  \ar[d] & \O_X(1) \ar[r]  \ar[d]^{\cong} & 0 \\ 0 \ar[r] & \Omega^1_X(1) \ar[r] \ar[d] & P^1(\O_X(1)) \ar[r] \ar[d] & \O_X(1)  \ar[r] & 0 \\ & 0 & 0 & & }.$$
Pick an integer $l \ge 0$. Tensoring the above diagram by $\O_X(l)$ and observing that 
$$P^1(\O_X(1)) \otimes \O_X(l) \cong P^1(\O_X(l+1))$$
by \cite[(2.2)]{ei}, we get the commutative diagram
\begin{equation}
\label{diag}
\xymatrix{& V \otimes H^0(\O_X(l)) \ar[d]^{f_l} \ar[dr]^{h_l} & \\ & H^0(P^1(\O_X(l+1))) \ar[r]^{\ \ \ g_l} \ar[d] & H^0(\O_X(l+1)) & \\ & H^1(N_X^*(l+1)).}
\end{equation}
Now we have that $H^1(N_X^*(l+1))=0$ since $N_X^*$ is aCM and $n \ge 2$. Hence $f_l$ is surjective for every $l \ge 0$ and so is $g_l$ by \cite[Prop.~2.3]{ei}. It follows by \eqref{diag} that $h_l$ is surjective for every $l \ge 0$.  
Moreover the commutative diagram
\begin{equation}
\label{diag2}
\xymatrix{V \otimes H^0(\O_{\P^{c+n}}(l)) \ar[r] \ar[d]^{\Id_V \otimes r_l} & H^0(\O_{\P^{c+n}}(l+1)) \ar[d]^{r_{l+1}} \\ V \otimes H^0(\O_X(l)) \ar@{->>}[r]^{h_l} & H^0(\O_X(l+1))}
\end{equation}
shows by induction that $r_l : H^0(\O_{\P^{c+n}}(l)) \to H^0(\O_X(l))$ is surjective for every $l \ge 0$, so that $X \subset \P^r$ is projectively normal, that is (ii). 

To see (iii) observe that, \eqref{eul} gives an exact sequence
\begin{equation}
\label{eul2}
\xymatrix{0 \ar[r] & H^0(\Omega^1_{\P^{c+n}}(1)_{|X}) \ar[r] & V \ar[r]^{\hskip -1cm f} & H^0(\O_X(1))}
\end{equation}
and $f$ is injective since $X$ is nondegenerate, hence $H^0(\Omega^1_{\P^{c+n}}(1)_{|X})=0$. Now $N_X^*(k)$ is aCM and therefore $H^1(N_X^*(1))=0$. Then the exact sequence
\begin{equation}
\label{nor2}
0 \to N_X^*(1) \to \Omega^1_{\P^{c+n}}(1)_{|X} \to \Omega^1_X(1) \to 0
\end{equation}
shows that $H^0(\Omega^1_X(1))=0$, hence, in particular $q(X)=h^0(\Omega^1_X)=0$. This proves (iii). Finally (iv) follows by Lemma \ref{s} since, $N_X^*(k)$ being Ulrich, we have that $H^0(N_X^*(k-1))=0$.
\end{proof}

\begin{lemma} {\rm (numerical conditions)} 
\label{num}

Let $X \subset \P^{c+n}$ be a smooth nondegenerate variety of dimension $n \ge 1$ and codimension $c \ge 1$ and degree $d$. If $N_X^*(k)$ is Ulrich, we have:
\begin{itemize}
\item[(i)] $[(k-2)c-2]d=(c+2)(g-1)$.
\item[(ii)] $c \ge 2$. 
\item[(iii)] $k \ge 3$.
\end{itemize}
\end{lemma}
\begin{proof}
By hypothesis $N_X^*(k)$ is Ulrich. By \eqref{eul} and \eqref{nor} we see that $c_1(N_X^*(k))=-K_X-(c+n+1-kc)H$. Hence Lemma \ref{ulr}(i) implies
$$-(K_X+(c+n+1-kc)H)H^{n-1} = \frac{c}{2}\left(K_X H^{n-1} + (n+1)d\right)$$
and this gives $K_XH^{n-1}=(2k-n-3-\frac{4(k-1)}{c+2})d$. But also $K_X H^{n-1} = 2(g-1)-(n-1)d$ and we get (i). As for (ii), if $c=1$ then $N_X^*(k) = \O_X(k-d)$ and Lemma \ref{ulr}(vi) gives that $X \subset \P^r$ is a linear space, a contradiction. This proves (ii). To see (iii), since $X$ is nondegenerate, we get from \eqref{eul} that $H^0({\Omega^1_{\P^{c+n}}}_{|X}(1))=0$. Now \eqref{nor} gives that $h^0({\Omega^1_{\P^{c+n}}}_{|X}(k)) \ge h^0(N_X^*(k))>0$ by Lemma \ref{ulr}(iii). Hence $k \ge 2$. But if $k=2$ then (i) gives that $g=0$ and $d=\frac{c+2}{2}$. As it is well known, $d \ge c+1$, giving a contradiction. Thus (iii) holds. 
\end{proof}

\section{Properties of the surface section}

We deduce here some very useful properties of the surface section of some $X \subset \P^r$ such that $N_X^*(k)$ is Ulrich. We assume that ${\rm char}(F)=0$.

\begin{lemma} 
\label{sur}

Let $X \subset \P^{c+n}$ be a smooth nondegenerate variety of dimension $n \ge 2$ and codimension $c \ge 1$. If $N_X^*(k)$ is Ulrich, the following inequality holds for the surface section $S$:
\smallskip

$\chi(\O_S) \ge \frac{d}{2(c+2)(c+3)(3c+4)(c+12)}[(3c^4+43c^3+86c^2+24c)k^2-(15c^4+229c^3+624c^2+432c)k+$ 

$\hskip 6cm +18c^4+296c^3+1070c^2+1368c+576]$.
\end{lemma}
\begin{proof}
Note that $N_{S/\P^{c+2}}^*(k)$ is Ulrich by Lemma \ref{sezi}. Therefore, Lemma \ref{coh}(iii) implies that
\begin{equation}
\label{irreg}
q(S)=0.
\end{equation}
Next, computing the Chern classes of $N_{S/\P^{c+2}}^*(k)$ and applying \cite[(2.2)]{c}, we get
\begin{equation}
\label{che}
cK_S^2+(c+12)c_2(S)=\frac{6d}{c+2}[c(c+2)k^2-c(5c+12)k+6c^2+20c+12].
\end{equation}
Now note that $N_{S/\P^{c+2}}^*(k)$ is semistable by \cite[Thm.~2.9]{ch}, hence so is $N_{S/\P^{c+2}}^*$ and since $\rk(N_{S/\P^{c+2}}^*)=c$, the Bogomolov inequality is
\begin{equation}
\label{bogo}
d(c+3)(4k-7-\frac{8(k-1)}{c+2})+(c+1)K_S^2 \ge (2c)c_2(S).
\end{equation} 
Then \eqref{che} and \eqref{bogo} give
\begin{equation}
\label{k2}
K_S^2 \ge \frac{d}{(c+2)(c+3)(3c+4)}[(12c^3+24c^2)k^2-(64c^3+204c^2+144c)k+79c^3+351c^2+486c+216].
\end{equation}
Finally the inequality on $\chi(\O_S)$ in the statement follows by \eqref{che}, \eqref{k2} and Noether's formula. 
\end{proof}

\section{Proofs of main theorems}

We start by proving Theorem \ref{sh}.

\renewcommand{\proofname}{Proof of Theorem \ref{sh}}
\begin{proof}
Since $N_C^*(k)$ is Ulrich, it follows that $c \ge 2$ by Lemma \ref{num}(ii) and $k \le s(C)$ by Lemma \ref{coh}(iv). Also Lemma \ref{coh}(i) gives $H^1(\O_C(k-2))=0$. Therefore we have that $H^0(\I_{C/\P^{c+1}}(k-1))=H^1(\O_C(k-1))=0$ and the exact sequence 
$$0 \to \I_{C/\P^{c+1}}(k-1) \to \O_{\P^{c+1}}(k-1) \to \O_C(k-1) \to 0$$
together with Riemann-Roch, shows that
\begin{equation}
\label{ine}
\binom{k+c}{c+1} = h^0(\O_{\P^{c+1}}(k-1)) \le h^0(\O_C(k-1))=d(k-1)-g+1.
\end{equation}
Also, $g-1=\frac{(k-2)c-2}{c+2}d$ by Lemma \ref{num}(i) and replacing in \eqref{ine} we get \eqref{bd}. 

Finally, sharpness for $c=2$ and $k \equiv 1, 3 \ ({\rm mod} \ 6)$ follows by \cite[Examples, p.~88]{be}, see Example \ref{sottoc}.
\end{proof}
\renewcommand{\proofname}{Proof}

Next, we prove Theorem \ref{n<2}.

\renewcommand{\proofname}{Proof of Theorem \ref{n<2}}
\begin{proof}
Suppose that $n \ge 2$. In order to simplify the calculations we set
$$A=(3c^4+43c^3+86c^2+24c)k^2-(15c^4+229c^3+624c^2+432c)k+18c^4+296c^3+1070c^2+1368c+576$$
so that it follows by Lemma \ref{sur} that
\begin{equation}
\label{chi}
\chi(\O_S) \ge \frac{dA}{2(c+2)(c+3)(3c+4)(c+12)}.
\end{equation}
Now Lemma \ref{sezi} implies that $N_{S/\P^{c+2}}^*(k)$ is Ulrich. Hence $k \le s(S)$ by Lemma \ref{coh}(iv), $H^2(\O_S(k-2))=0$  by Lemma \ref{coh}(i) and $S \subset \P^{c+2}$ is projectively normal by Lemma \ref{coh}(ii). Therefore, we have that 
$$H^0(\I_{S/\P^{c+2}}(k-2))=H^1(\I_{S/\P^{c+2}}(k-2))=H^2(\O_S(k-2))=0$$
and the exact sequence 
$$0 \to \I_{S/\P^{c+2}}(k-2) \to \O_{\P^{c+2}}(k-2) \to \O_S(k-2) \to 0$$
together with Riemann-Roch and Lemma \ref{num}(i), shows that
\begin{equation}
\label{in}
\begin{aligned}
\binom{k+c}{c+2} = h^0(\O_{\P^{c+2}}(k-2)) = h^0(\O_S(k-2))=\chi(\O_S(k-2))+h^1(\O_S(k-2)) \ge \\
\ge \chi(\O_S)+\frac{d(k-2)}{2}\big(3-k+\frac{4(k-1)}{c+2}\big).
\end{aligned}
\end{equation}
Setting
$$B=A+(k-2)(c+3)(3c+4)(c+12)[(3-k)(c+2)+4k-4]=2(c+12)(k-1)(8ck+12k+5c^2+7c)$$
we see, using \eqref{chi}, that \eqref{in} implies
\begin{equation}
\label{in2}
\binom{k+c}{c+2} \ge \frac{dB}{2(c+2)(c+3)(3c+4)(c+12)}=\frac{d(k-1)(8ck+12k+5c^2+7c)}{(c+2)(c+3)(3c+4)}.
\end{equation}
Since $k \ge 3$ by Lemma \ref{num}(iii), we see that \eqref{in2} gives 
$$d \le \frac{(c+2)(c+3)(3c+4)}{(k-1)(8ck+12k+5c^2+7c)}\binom{k+c}{c+2}.$$ 
Since $N_X^*(k)$ is Ulrich, it follows by Lemma \ref{sezi} that $N_{C/\P^{c+1}}^*(k)$ is Ulrich, hence, using Theorem \ref{sh}, we find that
$$\frac{c+2}{2k+c}\binom{k+c}{c+1} \le d \le \frac{(c+2)(c+3)(3c+4)}{(k-1)(8ck+12k+5c^2+7c)}\binom{k+c}{c+2}$$
that is equivalent to $2c(c+1)(c+k+1) \le 0$, a contradiction.
\end{proof}
\renewcommand{\proofname}{Proof}

\section{Curves}

In this section we construct some examples of curves $C \subset \P^{c+1}$ such that $N_C^*(k)$ is Ulrich. 

First, we give a reinterpretation of two known cases.

\begin{example} 
\label{elli}
For every integer $d \ge 5$ there exists a smooth elliptic curve $C \subset \P^3$ such that $N_C^*(3)$ is Ulrich. 

In fact, it follows by \cite[Prop.~\S 8, page 278]{el} and \cite[Thm.~2(b)]{eh} that for every integer $d \ge 5$ there exists a smooth elliptic curve $C \subset \P^3$ such that $H^0(N_C(-2))=0$. Since $\chi(N_C(-2))=0$ we get that $N_C(-2)$ is Ulrich and then also $N_C^*(3)$ is Ulrich by Lemma \ref{ulr}(ii).
\end{example}

\begin{example}
\label{sottoc}
There are many subcanonical curves $X \subset \P^3$ with $H^0(N_X(-2))=0$ (see for example \cite[Examples, p.~88]{be}), hence with $N_X(-1)$ Ulrich. Therefore, since $K_X=eH$ for some $e \in \Z$, we also have by Lemma \ref{ulr}(ii), that $N_X^*(e+3)$ is Ulrich. 

We now describe the curves in \cite[Examples, p.~88]{be}. Let $h \in \Z$ such that $h \ge 1$, let $c_2=h(3h+2)$ and let $t=3h+1$ (resp. $h \ge 0, c_2=3h^2+4h+1$ and $t=3h+2$). Then in loc. cit, there are examples of smooth curves $X \subset \P^3$ with $N_X^*(e+3)$ Ulrich, $d=t^2+c_2$ and $g=(t-2)d+1$. We have $e=\frac{2g-2}{d}=2t-4$, hence $N_X^*(k)$ is Ulrich with $k=e+3=2t-1$. In the first case, $t=3h+1$, we have $h=\frac{k-1}{6}, k \equiv 1 ({\rm mod} \ 6)$ and $d =\frac{1}{3}(k^2+2k)$. In the second case, $t=3h+2$, we have $h=\frac{k-3}{6}, k \equiv 3 ({\rm mod} \ 6)$ and again $d =\frac{1}{3}(k^2+2k)$. In particular they have unbounded $k$.
\end{example}

The examples above are all subcanonical curves. In order to construct non-subcanonical ones, we will proceed below by degeneration applying the techniques of interpolation as in \cite{aly, lv}. For the reader's convenience, we recall some definitions and notation about elementary modifications of vector bundles (see \cite{aly, lv}).

\begin{defi}
Let $Z$ be a scheme, let $\E$ be a vector bundle on $Z$, let $D$ be an effective Cartier divisor on $Z$ and let $\F$ be a subbundle of $\E_{|D}$. Then
$$\E \hskip -.1cm \xymatrix{[D \ar[r(0.6)]^{\hskip .3cm -} & \hskip -.4cm \F]} = \Ker\{\E \to \E_{|D}/\F\}, \ \ \E \hskip -.1cm \xymatrix{[D \ar[r(0.6)] & \hskip -.4cm \F]} \hskip -.2cm = \E \hskip -.1cm \xymatrix{[D \ar[r(0.6)]^{\hskip .3cm -} & \hskip -.4cm \F]}, \ \ \E \hskip -.1cm \xymatrix{[D \ar[r(0.6)]^{\hskip .3cm +} & \hskip -.4cm \F]} \hskip -.2cm = \E \hskip -.1cm\xymatrix{[D \ar[r(0.6)]^{\hskip .3cm -} & \hskip -.4cm \F]} \hskip -.1cm (D).$$
\end{defi}

\begin{defi}
Let $X$ be a smooth curve, let $f: X \to \P^r$ be an unramified morphism and let $N_f = \Ker\{f^*\Omega_{\P^r} \to \Omega_X\}^\vee$ be the normal bundle of the map $f$. Given a subspace $\Lambda \subset \P^r$ of dimension $\lambda$, consider the projection $\pi_{\Lambda} : \P^r \dashrightarrow \P^{r-\lambda-1}$. Let $U$ be the open subset of $X \setminus X \cap \Lambda$ where $\pi_{\Lambda} \circ f$ is unramified. Then $N_{f \to \Lambda}$ is be the unique subbundle of $N_f$ whose restriction to $U$ is $\Ker \{(N_f)_{|U} \to (N_{\pi_{\Lambda} \circ f})_{|U}\}$. Moreover, 
$$N_f \hskip -.1cm \xymatrix{[D \ar[r(0.6)]^{\hskip .3cm +} & \hskip -.4cm \Lambda]} \hskip -.2cm := N_f \hskip -.1cm \xymatrix{[D \ar[r(0.5)]^{\hskip .3cm +} & \hskip -.2cm N_{f \to \Lambda}]} \ \hbox{and} \ N_f \hskip -.1cm \xymatrix{[D \ar[r(0.6)] & \hskip -.4cm \Lambda]} \hskip -.2cm := N_f \hskip -.1cm \xymatrix{[D \ar[r(0.5)]^{\hskip .3cm -} & \hskip -.3cm N_{f \to \Lambda}]}.$$ 
When $f$ is an embedding, we set $N_X \hskip -.1cm \xymatrix{[D \ar[r(0.6)] & \hskip -.4cm \Lambda]} \hskip -.1cm =N_f \hskip -.1cm \xymatrix{[D \ar[r(0.6)] & \hskip -.4cm \Lambda]}$. 
\end{defi}

Next, we prove Theorem \ref{nuovo2}. 

\renewcommand{\proofname}{Proof of Theorem \ref{nuovo2}(i)}
\begin{proof}
From $d=2g-2$ we deduce that $g \ge 2$, hence, since $X$ is nonspecial, $g-1=h^0(\O_C(1)) \ge 3$. Thus $g \ge 4$ and if equality holds, then $C$ is a plane curve of degree $6$, a contradiction. Therefore $g \ge 5$.

Let $E \subset \P^3$ be a general elliptic curve of degree $g-1$ and let $L_i$ be general chords of $E$ with $L_i \cap E = \{p_i, q_i\}, 1 \le i \le g-1$. Let $Y=L_1 \cup \ldots L_{g-1}$ and let 
$$X'=E \cup Y \subset \P^3.$$ 
\renewcommand{\proofname}{Proof}

\begin{claim}
\label{1}
With the above notation, we have
$${N_{X'}}_{|_Y} \otimes \omega_{X'}(-3) \cong \bigoplus\limits_{i=1}^{g-1} \O_{L_i}(-1)^{\oplus 2}.$$
\end{claim}
\begin{proof}
For each $i \in \{1, \ldots, g-1\}$ set $Z_i = E \cup L_1 \cup \ldots L_{i-1} \cup L_{i+1} \cup \ldots \cup L_{g-1}$. By \cite[Prop.~3.5]{lv}, if we choose points $r_i \in T_{p_i}Z_i \setminus \{p_i\}, r'_i \in T_{q_i}Z_i \setminus \{q_i\}$, we have
\begin{equation}
\label{rette}
{N_{X'}}_{|_{L_i}} \cong N_{L_i} \hskip -.1cm \xymatrix{[p_i \ar[r(0.6)]^{\hskip .2cm +} & \hskip -.4cm r_i]} \hskip -.2cm \xymatrix{[q_i \ar[r(0.6)]^{\hskip .2cm +} & \hskip -.4cm r'_i]} = N_{L_i} \hskip -.1cm \xymatrix{[p_i \ar[r(0.6)]^{\hskip .2cm -} & \hskip -.4cm r_i]} \hskip -.2cm \xymatrix{[q_i \ar[r(0.6)]^{\hskip .2cm -} & \hskip -.4cm r'_i]}(p_i+q_i).
\end{equation}
On the other hand, let $f : L_i \to \P^3$ be the embedding. Since $r_i \not\in L_i$, we have that $f \circ \pi_{r_i} : L_i \to \P^2$
is an embedding, and therefore $N_{f \to \{r_i\}}=\O_{L_i}(1)$ and $N_{L_i}/N_{f \to \{r_i\}}$ is a line bundle on $L_i$.
Therefore 
$$(N_{L_i})_{|\{p_i\}}/(N_{f \to \{r_i\}})_{|\{p_i\}} \cong \O_{\{p_i\}}$$
and, similarly,
$$(N_{L_i})_{|\{q_i\}}/(N_{f \to \{r'_i\}})_{|\{q_i\}} \cong \O_{\{q_i\}}.$$
Hence
$$N_{L_i} \hskip -.1cm \xymatrix{[p_i \ar[r(0.6)]^{\hskip .2cm -} & \hskip -.4cm r_i]} \hskip -.2cm \xymatrix{[q_i \ar[r(0.6)]^{\hskip .2cm -} & \hskip -.4cm r'_i]} =\Ker \{N_{L_i} \to \O_{\{p_i\}} \oplus \O_{\{q_i\}}\}=\Ker \{\O_{\P^1}(1)^{\oplus 2} \to \O_{\{p_i\}} \oplus \O_{\{q_i\}}\}\cong \O_{\P^1}^{\oplus 2}.$$
It follows by \eqref{rette} that 
$${N_{X'}}_{|_{L_i}} \cong \O_{\P^1}(p_i+q_i)^{\oplus 2} = \O_{\P^1}(2)^{\oplus 2}$$
and therefore
$${N_{X'}}_{|_Y} \otimes \omega_{X'}(-3) \cong \bigoplus\limits_{i=1}^{g-1} \O_{L_i}(-1)^{\oplus 2}.$$
\end{proof}

\begin{claim}
\label{5}
With the above notation, we have that $H^0({N_{X'}}_{|_E}(-3))=0$.
\end{claim}
\begin{proof}
By \cite[Prop.~3.5]{lv}, since $q_i \in L_i \setminus \{p_i\}, p_i \in L_i \setminus \{q_i\}$, we have
$${N_{X'}}_{|_E} \cong N_E \hskip -.1cm \xymatrix{[p_1 \ar[r(0.6)]^{\hskip .2cm +} & \hskip -.4cm q_1]} \cdots \xymatrix{[p_{g-1} \ar[r(0.6)]^{\hskip .4cm +} & \hskip -.4cm q_{g-1}]} \xymatrix{[q_1 \ar[r(0.6)]^{\hskip .2cm +} & \hskip -.4cm p_1]} \cdots \xymatrix{[q_{g-1} \ar[r(0.6)]^{\hskip .4cm +} & \hskip -.4cm p_{g-1}]}.$$
Next, specialize cyclically $q_1 \to p_2, q_2 \to p_3, \ldots, q_{g-1} \to p_1$. Then, we get a flat family of vector bundles with central fiber
$$N_E \hskip -.1cm \xymatrix{[p_1 \ar[r(0.6)]^{\hskip .2cm +} & \hskip -.4cm p_2]} \cdots \xymatrix{[p_{g-1} \ar[r(0.6)]^{\hskip .4cm +} & \hskip -.4cm p_1]}\xymatrix{[p_2 \ar[r(0.6)]^{\hskip .2cm +} & \hskip -.4cm p_1]} \cdots \xymatrix{[p_1 \ar[r(0.6)]^{\hskip .4cm +} & \hskip -.3cm p_{g-1}]}.$$
We now claim that
\begin{equation}
\label{ne}
N_E \hskip -.1cm \xymatrix{[p_1 \ar[r(0.6)]^{\hskip .2cm +} & \hskip -.4cm p_2]} \cdots \xymatrix{[p_{g-1} \ar[r(0.6)]^{\hskip .4cm +} & \hskip -.4cm p_1]}\xymatrix{[p_2 \ar[r(0.6)]^{\hskip .2cm +} & \hskip -.4cm p_1]} \cdots \xymatrix{[p_1 \ar[r(0.6)]^{\hskip .4cm +} & \hskip -.3cm p_{g-1}]} \cong N_E(p_1+ \ldots + p_{g-1}).
\end{equation}
To see \eqref{ne}, let $f : E \to \P^3$ be the embedding. For every point $q \in E$, any modification of type $p$ pointing towards the linear space $\{q\}$, is obtained  by taking a quotient $(N_E)_{|\{p\}}/(N_{f \to \{q\}})_{|\{p\}}$. Since $N_E/N_{f \to \{q\}}$ is a line bundle on $E$, we have that 
$$(N_E)_{|\{p\}}/(N_{f \to \{q\}})_{|\{p\}} \cong \O_{\{p\}}.$$
Therefore, 
$$N_E \hskip -.1cm \xymatrix{[p_1 \ar[r(0.6)]^{\hskip .2cm -} & \hskip -.4cm p_2]} \cdots \xymatrix{[p_{g-1} \ar[r(0.6)]^{\hskip .4cm -} & \hskip -.4cm p_1]}\xymatrix{[p_2 \ar[r(0.6)]^{\hskip .2cm -} & \hskip -.4cm p_1]} \cdots \xymatrix{[p_1 \ar[r(0.6)]^{\hskip .4cm -} & \hskip -.3cm p_{g-1}]} =$$ 
$$\hskip .9cm =\Ker \{N_E \to \O_{\{p_1\}} \oplus \ldots \O_{\{p_{g-1}\}} \oplus \O_{\{p_2\}} \ldots \O_{\{p_{g-1}\}} \oplus \O_{\{p_1\}}\}=$$
$$\hskip -.9cm = \Ker \{N_E \to \O_{\{p_1\}}^{\oplus 2} \oplus \ldots \O_{\{p_{g-1}\}} ^{\oplus 2}\}= N_E(-p_1-\ldots p_{g-1}).$$
We deduce that
$$N_E \hskip -.1cm \xymatrix{[p_1 \ar[r(0.6)]^{\hskip .2cm +} & \hskip -.4cm p_2]} \cdots \xymatrix{[p_{g-1} \ar[r(0.6)]^{\hskip .4cm +} & \hskip -.4cm p_1]}\xymatrix{[p_2 \ar[r(0.6)]^{\hskip .2cm +} & \hskip -.4cm p_1]} \cdots \xymatrix{[p_1 \ar[r(0.6)]^{\hskip .4cm +} & \hskip -.3cm p_{g-1}]}=$$
$$=N_E \hskip -.1cm \xymatrix{[p_1 \ar[r(0.6)]^{\hskip .2cm -} & \hskip -.4cm p_2]} \cdots \xymatrix{[p_{g-1} \ar[r(0.6)]^{\hskip .4cm -} & \hskip -.4cm p_1]}\xymatrix{[p_2 \ar[r(0.6)]^{\hskip .2cm -} & \hskip -.4cm p_1]} \cdots \xymatrix{[p_1 \ar[r(0.6)]^{\hskip .4cm -} & \hskip -.3cm p_{g-1}]}(2p_1+ \ldots 2 p_{g-1}) \cong N_E(p_1+ \ldots + p_{g-1})$$
and this proves \eqref{ne}.

Thus we have proved that ${N_{X'}}_{|_E}(-3)$ specializes to $N_E(p_1+ \ldots + p_{g-1})(-3)$ and to prove the claim it suffices to show that $H^0(N_E(p_1+ \ldots + p_{g-1})(-3))=0$. Now, $\chi(N_E(p_1+ \ldots + p_{g-1})(-3))=0$ and $\O_E(p_1+ \ldots + p_{g-1})(-3)$ is a general line bundle of degree $-2(g-1)$, hence $H^0(N_E(p_1+ \ldots + p_{g-1})(-3))=0$ by interpolation \cite[Thm.~1.2]{aly},  \cite[Thm.~1.4]{lv}.
\end{proof}
We now conclude the proof of Theorem \ref{nuovo2}(i). We have that $X' = E \cup Y$ is a nodal nonspecial curve with $\deg X'=2g-2$ and $p_a(X')=g$. In particular $X'$ is smoothable in $\P^3$. Moreover, the exact sequence
$$0 \to {N_{X'}}_{|_E}(-3) \to N_{X'} \otimes \omega_{X'}(-3) \to {N_{X'}}_{|_Y} \otimes \omega_{X'}(-3) \to 0$$
and Claims \ref{1} and \ref{5} show that $H^0(N_{X'} \otimes \omega_{X'}(-3))=0$. 
Then a general smoothing $X \subset  \P^3$ of $X'$ is such that $\deg X = 2g-2, g(X)=g$ and $N_X^*(4)$ is Ulrich. Finally $X$ is not subcanonical, for otherwise we would have that $K_C=H$, contradicting the fact that $X$ is nonspecial.
\end{proof}
\renewcommand{\proofname}{Proof}

Finally, we prove Theorem \ref{nuovo2}(ii).

\renewcommand{\proofname}{Proof of Theorem \ref{nuovo2}(ii)}
\begin{proof}
Let $\overline C \subset \P^3$ be a general nonspecial curve of degree $3g-3$ and genus $g \ge 3$. Since $\overline C$ is nonspecial, we can embed it as a nonspecial curve $C \subset \P^4$ of degree $3g-2$ and genus $g$,  so that $\overline C$ is just the projection from a general point $q \in C$\footnote{For the construction of the curves $C$ and $\overline C$ we could also have appealed to \cite[Thm.~1]{eh2}, \cite{cpjllv}.}.
%Let $C$ a general curve of genus $g \ge 3$. Note that $\rho(g,4,3g-2)=11g-30>0$, hence a general $g^4_{3g-2}$ exists and is very ample by \cite[Thm.~1]{eh2}, \cite{cpjllv}. Let $g^4_{3g-2} = (V, L)$, so that we have an embedding $C \subset \P^4 = \P V$. 
Let $p_i, 1 \le i \le 2g-3$ be general points on $C$ and let $L_i$ be general $1$-secant lines, each meeting $C$ at $p_i$.  Let $Y=L_1 \cup \ldots L_{2g-3},$ let $q_i \in L_i \setminus \{p_i\}$ and let 
$$X'=C \cup Y \subset \P^4.$$ 
Note that $X'$ is nonspecial, hence we can choose a flat family $\mathcal X \to \Delta$, with smooth total space, having as a general fiber $X$, a general curve in $\P^4$ of genus $g$ and degree $5g-5$ and special fiber $X'$. Since $C$ is a Cartier divisor on $\mathcal X$ we have that $\omega_{\mathcal X/\Delta}(-2)(C)$ is a line bundle on $\mathcal X$. Note that, since $C+Y$ is a fiber, we have that $C_{|C}=-Y_{|C}$. Setting 
$$\L= \omega_{\mathcal X/\Delta}(-2)(C)_{|X'}$$
we have that 
$$\L_{|C}=\omega_C(Y \cap C)(-2)(-Y_{|C})=\omega_C(-2)$$ 
and, for each $i$,
$$\L_{|L_i}=\omega_{L_i}(L_i \cap C)(-2)(L_i \cap C)=\O_{L_i}(-2).$$
\renewcommand{\proofname}{Proof}
Now we have
\begin{claim} 
\label{6}
Let $\E =N_C(p_1+ \ldots + p_{2g-3}) \hskip -.1cm \xymatrix{[2p_1 \ar[r(0.6)] & \hskip -.4cm q_1]} \cdots \xymatrix{[2p_{2g-3} \ar[r(0.6)] & \hskip -.4cm q_{2g-3}]} \otimes \omega_C(-2)$. If $H^0(\E)=0$, then $H^0(N_{X'} \otimes \L)=0$.
\end{claim}
\begin{proof}
Note that $\chi(\E)=0$. In fact, we have that
$$\chi(N_C(p_1+ \ldots + p_{2g-3}))=20g-18$$
so that, using \cite[Prop.~3.3(a)]{aly}, we get
$$\chi(N_C(p_1+ \ldots + p_{2g-3}) \hskip -.1cm \xymatrix{[2p_1 \ar[r(0.6)] & \hskip -.4cm q_1]} \cdots \xymatrix{[2p_{2g-3} \ar[r(0.6)] & \hskip -.4cm q_{2g-3}]})=\chi(N_C(p_1+ \ldots + p_{2g-3}))-4(2g-3)=12g-6$$
and therefore
$$\chi(\E)=12g-6+3 \deg(\omega_C(-2))=0.$$
Since $H^0(\E)=\chi(\E)=0$, we have that $H^1(\E)=0$ and for every divisor $D>0$ on $C$ we get that $H^0(\E(-D)) \subseteq H^0(\E)=0$, hence $\E$ satisfies interpolation. As in \cite[Lemma 8.5]{aly} (see also \cite[Lemma 5.2]{lv}), we get that $N_{X'} \otimes \L$ satisfies interpolation, so in particular $H^1(N_{X'} \otimes \L)=0$. On the other hand, $\chi(N_{X'} \otimes \L)=\chi(N_{X'})+3\deg \L=0$, and therefore  $H^0(N_{X'} \otimes \L)=0$.
\end{proof}
In order to show that $H^0(\E)=0$, %we choose a general point $q \in C$. Since $(V(-q), L(-q))$ is a general $g^3_{3g-3}$ on $C$, it is very ample by \cite[Thm.~1]{eh2}, \cite{cpjllv}, and defines an embedding of $C$ in $\P^3$ whose image $\overline C \subset \P^3$ is just the projection of $C$ from $q$. 
we specialize each $q_i$ to $q$, so that $\E$ specializes to
$$N_C(p_1+ \ldots + p_{2g-3}) \hskip -.1cm \xymatrix{[2p_1 \ar[r(0.6)] & \hskip -.4cm q]} \cdots \xymatrix{[2p_{2g-3} \ar[r(0.6)] & \hskip -.4cm q]} \otimes \omega_C(-2) \cong $$
$$\hskip -.4cm \cong N_C(p_1+ \ldots + p_{2g-3}) \hskip -.1cm \xymatrix{[2p_1+ \ldots + 2p_{2g-3} \ar[r(0.8)] & \hskip -.4cm q]} \hskip -.15cm \otimes \hskip .1cm \omega_C(-2).$$
We have
\begin{claim} 
\label{4}
Set $\E_1 = N_C(p_1+ \ldots + p_{2g-3})\xymatrix{[2p_1+ \ldots + 2p_{2g-3} \ar[r(0.8)] & \hskip -.4cm q]} \otimes \omega_C(-2).$
Then there is an exact sequence 
\begin{equation}
\label{e1}
0 \to \omega_C(-1)(2q+p_1+ \ldots + p_{2g-3}) \to \E_1 \to N_{\overline C/\P^3}(q-p_1- \ldots - p_{2g-3}) \otimes \omega_C(-2) \to 0.
\end{equation}
\end{claim}
\begin{proof}
Set $D=p_1+ \ldots + p_{2g-3}$. By \cite[page 10]{lv}, the pointing (towards $q$) bundle exact sequence \cite[(3.4)]{lv}, becomes
$$0 \to \O_C(1)(2q) \to N_C \to N_{\overline C/\P^3}(q) \to 0$$
hence, tensoring by $\O_C(D)$, we find an 
exact sequence
\begin{equation}
\label{succ}
0 \to \O_C(1)(2q+D) \to N_C(D) \to N_{\overline C/\P^3}(q+D) \to 0.
\end{equation}
Applying the elementary modification $\xymatrix{[2D \ar[r(0.6)] & \hskip -.5cm q]}$ \hskip -.2cm we get the exact sequence
$$0 \to \O_C(1)(2q+D) \to N_C(D)\xymatrix{[2D \ar[r(0.6)] & \hskip -.5cm q]}  \to N_{\overline C/\P^3}(q+D)(-2D) \to 0$$
and tensoring by $\omega_C(-2)$ we get \eqref{e1}.
\end{proof}
Since $\omega_C(-1)(2q+p_1+ \ldots + p_{2g-3})$ is a general line bundle of degree $g-1$ on $C$, we have that 
$$H^0(\omega_C(-1)(2q+p_1+ \ldots + p_{2g-3}))=0.$$ 
Also, $\omega_C(-2)(q-p_1- \ldots - p_{2g-3})$ a general line bundle of degree $-6(g-1)$ on $C$ and 
$$\chi(N_{\overline C/\P^3}(q-p_1- \ldots - p_{2g-3})  \otimes \omega_C(-2))=0$$ 
hence 
$$H^0(N_{\overline C/\P^3}(q-p_1- \ldots - p_{2g-3})  \otimes \omega_C(-2))=0$$ by interpolation \cite[Thm.~1.2]{aly}, \cite[Thm.~1.4]{lv}.
Then \eqref{e1} implies that $H^0(\E_1)=0$ and therefore, since $\E$ specializes to $\E_1$, also that $H^0(\E)=0$. It follows by Claim 
\ref{6} that $H^0(N_{X'} \otimes \L)=0$ and then $H^0(N_X \otimes \omega_X(-2))=0$ by semicontinuity. Hence $N_X^*(3)$ is Ulrich and, clearly, $X$ is not subcanonical.
\end{proof}

\end{document}